\documentclass{amsart}

\usepackage{amssymb}

\newcommand{\vvert}{|\!|\!|}

\newcommand{\RR}{{\mathbb R}}
\newcommand{\CC}{{\mathbb C}}
\newcommand{\NN}{{\mathbb N}}
\newcommand{\ZZ}{\mathbb Z}

\def\B{{\mathcal B}}

\def\L{{\mathcal L}} 
\def\M{{\mathcal M}}
 
\def\P{{\mathcal P}}
\def\R{{\mathcal R}}

\numberwithin{equation}{section}

\newtheorem{theo}{Theorem}    
\newtheorem{prop}[theo]{Proposition}  
\newtheorem{coro}[theo]{Corollary}  
\newtheorem{lemma}[theo]{Lemma}  
\newtheorem{defi}[theo]{Definition}  

\begin{document}  
\title  
{  
Continuous and measurable eigenfunctions of linearly recurrent 
dynamical Cantor systems
}  

\author{Maria Isabel Cortez}
\address{Departamento de Ingenier\'{\i}a
Matem\'atica, Universidad de Chile, Casilla 170/3 correo 3,
Santiago, Chile.}
\email{mcortez@dim.uchile.cl} 

\author{Fabien  Durand}  
\address{Laboratoire Ami\'enois
de Math\'ematiques Fondamentales  et
Appliqu\'ees, CNRS-UMR 6140, Universit\'{e} de Picardie
Jules Verne, 33 rue Saint Leu, 80000 Amiens, France.}
\email{fdurand@u-picardie.fr}

\author{Bernard Host}  
\address{\'Equipe d'Analyse et Math\'ematiques Appliqu\'ees, CNRS-UMR 8050,
Universit\'e de Mar\-ne-la-Vall\'ee, 93166 Noisy-le-Grand, France.}
\email{host@math.univ-mlv.fr}

\author{Alejandro Maass}
\address{Departamento de Ingenier\'{\i}a
Matem\'atica, Universidad de Chile
and Centro de Modelamiento Ma\-te\-m\'a\-ti\-co,
UMR 2071 UCHILE-CNRS, Casilla 170/3 correo 3,
Santiago, Chile.}
\email{amaass@dim.uchile.cl}

\subjclass{Primary: 54H20; Secondary: 37B20 } 
\keywords{minimal Cantor systems, linearly recurrent dynamical 
systems, eigenvalues}

\begin{abstract} 

The class of linearly recurrent Cantor systems contains 
the substitution subshifts and some odometers. For substitution
subshifts measure--theoretical  and continuous eigenvalues are 
the same.
It is natural to ask  whether this rigidity property remains true for 
the class of linearly recurrent Cantor systems. We give partial 
answers to this question.

\end{abstract}

\maketitle  
\markboth{Maria Isabel Cortez, Fabien Durand, Bernard Host, Alejandro
Maass}{Eigenfunctions of linearly recurrent dynamical Cantor systems}

\section{Introduction}

Let $(X,T)$ be a topological dynamical system and 
$\mu$ a $T$--invariant probability measure.
When a measure-theoretical  eigenvalue $\lambda \in \CC$ of the  
system, that is   
$f\circ T=\lambda f$ for some $f \in L^2(\mu)\setminus\{0\}$, is
associated to 
a continuous eigenfunction $f:X \to \CC$ ?
In this paper we are interested in 
conditions on minimal dynamical Cantor systems 
that answer  
this question. Our motivation comes from 
 \cite{Ho} where it is proved that all eigenvalues of minimal 
substitution 
subshifts are associated to 
a continuous eigenfunction. 
Such a question also  appears in \cite{NR} where the 
authors show  that generically interval exchange transformations are 
not 
topologically weakly mixing (i.e., they do not have non 
trivial continuous eigenfunctions) 
and where they ``fully expect'' the same holds 
for (measure-theoretical) weak mixing (i.e., they do not 
have non trivial eigenfunctions). 
It is in general not true that all eigenvalues of a minimal dynamical 
system have a continuous eigenfunction as can be seen for some 
Toeplitz systems \cite{Iw,DL} and for some 
interval exchange transformations \cite{FHZ}.

In this paper we focus on linearly recurrent 
dynamical Cantor systems (also called linearly recurrent systems). 
They naturally extend the notion of substitution 
subshifts in the sense they share the same return time structure. 
Linearly recurrent subshifts were studied in 
\cite{DHS,Du1,Du2,Le}.

The paper is organized as follows. In Section~\ref{definitions} we 
define 
linearly recurrent systems by means of  nested sequence of 
Kakutani-Rokhlin 
partitions and obtain some general properties. In particular we prove 
that these systems 
are uniquely ergodic but are not strongly mixing. 

In the following section, when the dynamical system $(X,T)$ is 
linearly 
recurrent and  $\mu$ is a $T$--invariant probability measure 
we give a necessary 
condition for a complex number to be an eigenvalue. We also exhibit a 
sufficient condition for a complex 
number to be a continuous eigenvalue,
which involves the underlying matrix structure 
of the nested sequence of Kakutani-Rokhlin partitions defining 
$(X,T)$. 
This is used in the last section to prove for  natural probability 
spaces associated to families of linearly recurrent systems, 
and under a condition of ``hyperbolicity'',  that almost every system of such family has 
only continuous eigenvalues. We give in Section~\ref{sectionexample} several examples to 
illustrate the results of the paper.

\section{Definitions and background}

\label{definitions}

\subsection{Dynamical systems}

By a {\it topological dynamical system} we mean a couple $(X,T)$ 
where 
$X$ is a compact metric space and $T: X \to X$ is a homeomorphism. 
We say that it is a {\it Cantor system} if $ X $ is a Cantor space; 
that is, $ X $ 
has a countable basis of its topology which consists of closed and 
open sets
(clopen sets) and does not have isolated points. The topological
dynamical system $(X,T)$ is {\it minimal} 
whenever $X$ and the empty set are the only $T$-invariant closed 
subsets of $X$. 
We only deal here with  minimal Cantor systems.
A complex number $\lambda$ is a {\it continuous eigenvalue} 
of $(X,T)$ if there exists a continuous function $f : X\to \CC$, 
$f\not = 0$, such that 
$f\circ T = \lambda f$; $f$ is called a {\it continuous 
eigenfunction} 
associated to $\lambda$. If $(X,T)$ is minimal, then every continuous 
eigenvalue is of modulus 1 and every continuous eigenfunction has a constant 
modulus.

When $(X,T)$ is a topological dynamical system and $\mu$ is a 
$T$-invariant probability measure, i.e., $T\mu = \mu$,  defined on 
the Borel
$\sigma$-algebra $\B_X$ of $X$, we call the triple $(X,T,\mu )$ a 
{\it dynamical system}. We do not recall the definitions of 
ergodicity, weak mixing and strong mixing (see \cite{Wa} for example).
A complex number $\lambda$ is an {\it eigenvalue} 
of the dynamical system $(X,T,\mu)$ if there exists $f\in L^2 (\mu)$, 
$f\not = 0$, such that 
$f\circ T = \lambda f$, $\mu $-a.e.; $f$ is called an {\it 
eigenfunction} 
(associated to $\lambda$). If the system is ergodic, then every 
eigenvalue is of modulus 1, and every eigenfunction has a constant 
modulus. 
By abuse of language we will also say that an eigenvalue is continuous
when the associated eigenfunction is continuous.

In this paper we mainly consider topological dynamical systems 
$(X,T)$ which are uniquely ergodic, that is systems that admit a unique 
invariant probability measure; this measure is then ergodic.

\subsection{Partitions}

\label{subsecpart}

Sequences of partitions of a minimal Cantor 
system were used in \cite{HPS} to build a representation  
of the system as an adic transformation on an ordered 
Bratteli diagram. We recall some definitions and fix some notations 
we shall use along this paper.

Let $(X,T)$ be a minimal Cantor system.
A {\it clopen  Kakutani-Rokhlin partition} (CKR partition) is  a 
partition $\P$ of $X$ of the kind:
\begin{equation}
\label{eq:def-KR}
\P = \{ T^{-j} B_k ; 1\leq k\leq C , \  0 \leq  j< h_k \} 
\end{equation}
where $C$ is a positive integer, 
$B_1,\dots , B_C$ are clopen subsets of $X$ and 
$h_1,\dots,h_k$ are positive integers.  
For $1\leq k\leq C$, the $k$-th {\it tower} of $\P$ is 
the family $\{ T^{-j} B_k ; 0 \leq j <  h_k \}$, 
and the {\it base} of $\P$ is the set 
$B =\bigcup_{1\leq k\leq C}B_k$.
Let 
\begin{equation}
\label{eq:def-seq-KR}
\bigl(
\P (n)=
\{
T^{-j}B_{k} (n): 1\leq k\leq C(n),\ 0\le j<h_{k}(n)
\}
\ ; \ n\in\NN
\bigr)
\end{equation}
be a sequence of CKR  partitions. For every $n$ we write 
$B(n)$ for the base of $\P (n)$,  and we assume that $\P (0)$ is the 
trivial partition, that is $B (0)=X$, $C(0)=1$ and $h_{1} (0) = 1$.

We say that this sequence $(\P (n);n\in \NN)$ is \emph{nested} if for 
every 
$n\geq 0$ it satisfies:

\medskip

{\bf (KR1)} $B (n+1) \subset B (n)$ and

\medskip

{\bf (KR2)} $\P (n+1) \succeq \P (n)$; i.e., for all $A\in \P (n+1)$ 
there exists $A^{'}\in \P (n)$ such that $A\subset A^{'}$. 

\medskip

We consider mostly nested sequences of CKR partitions which 
satisfy also the properties: 

\medskip

{\bf (KR3)} $ \cap_{n=0}^{\infty} B (n) $ consists of a unique point;

\medskip

{\bf (KR4)} the sequence of partitions spans the topology of $X$;

\medskip

In \cite{HPS} it is proven that for each minimal Cantor system 
$(X,T)$ 
there exists a nested sequence of CKR partitions fulfilling
(KR1)-(KR4) (i.e., (KR1), (KR2), (KR3) and (KR4)) and the following 
conditions:

\medskip

{\bf (KR5)} for all $n\geq 1$, $1\leq k \leq C (n-1)$, $1\leq l\leq C 
(n)$, 
there exists $0 \leq j < h_{l} (n)$ such that $T^{-j} B_{l}(n) 
\subset B_{k}(n-1)$;

\medskip

{\bf (KR6)} for all $n\in \NN$, $B (n+1) \subset B_{1}(n)$.

\medskip

To such a sequence of partitions we 
associate a sequence of matrices $(M(n) ; n\ge 1 )$, where the 
matrix
$M (n) = (m_{l,k}(n); 1\leq l \leq C (n) , 1\leq k \leq C (n-1))$ is 
given by
$$
m_{l,k}(n) = 
\#
\{
0 \leq j < h_{l}(n) ; T^{-j} B_{l}(n) \subset B_{k}(n-1)
\}
.
$$

We notice that Property (KR5) is equivalent to the condition that
for every $n\ge 1$ the matrix $M(n)$ has positive entries.
As the sequence of partitions is nested, we get 
$$
h_{l}(n)=\sum_{k=1}^{C (n-1)} m_{l,k}(n)h_{k}(n-1) ,
\ \ 
n\geq 1,\ 1\leq l\leq C (n) .
$$

We rewrite this formula in a matrix form. 
For every $n\geq 0$, let $H(n)=(h_{l}(n) ; 1\leq l\leq C (n))^t$, 
that is the 
column vector with entries $h_{1}(n),h_{2}(n),\dots, h_{C (n)} (n)$. 
Then 
we have $H (n)=M (n)H (n-1)$ for $n>0$. For $n>m\geq 0$ we define 
$$
P(n,m)=M(n)M(n-1)\dots M(m+1) \hbox{ and } P(n)=P(n,1)\ .
$$
We have:
$$
P_{l,k}(n,m)=\#\bigl\{ j ; 0\leq j\leq h_l(n),\ 
T^{-j}B_l(n)\subset B_k(m)\bigr\}
$$
and
$$
P(n,m)H(m)= H(n)=P(n)H(1)\ .
$$

\subsection{Linearly recurrent systems}

\begin{defi}
A minimal Cantor system $(X,T)$ is {\it linearly recurrent} (with 
constant $L$) if there exists a nested sequence of 
CKR 
partitions 
$(\P(n)=\{T^{-j}B_{k}(n); 1\leq k\leq C(n), 0\le j<h_{k}(n)\};n\in\NN 
)$ 
satisfying (KR1)-(KR6) and

\medskip

{\bf (LR)}
there exists $L$ such that for all $(l,k) \in \{ 1, \dots , C(n) \} 
\times 
\{ 1 , \dots , C(n-1) \}$ and all $n\geq 1$ 
$$
 h_{ l}(n)\leq L \ h_{k}(n-1)\ .
$$
\end{defi}

The notion of linearly recurrent dynamical Cantor system 
(also called linearly recurrent system) is the extension of the 
concept
of linearly recurrent subshift introduced 
 in \cite{DHS}. Of course it can be proved that linearly recurrent
 subshifts are linearly recurrent 
Cantor systems (see \cite{Du1,Du2}). 
Examples of such systems are substitution subshifts \cite{DHS} and 
 Sturmian subshifts whose associated rotation number has a continued 
 fraction with bounded coefficients \cite{Du1,Du2}.

\begin{lemma}
\label{cnborne} 
Let $(X,T)$ be a linearly recurrent system 
and $(\P (n) ; n\in \NN )$ a sequence of CKR partitions satisfying 
Properties 
(KR1)-(KR6), and Property (LR) with constant $L$. Then:
\begin{enumerate}
\item
for every $n\in \NN$ we have $C(n) \leq L$;
\item
for every $n\in\NN$, $1\leq k\leq C(n)$ and $1\leq k'\leq C(n)$ we 
have
$ h_{ k}(n) \leq L \ h_{k'}(n) $.
\end{enumerate}
\end{lemma}

\begin{proof}
Property  (1)  follows directly from the hypotheses (KR5) and (LR). 
Indeed $ C(n) h_{i}(n) \leq h_{ 1}(n+1) \leq L h_{i}(n)$ where 
$h_{i}(n) = \min \{ h_{k}(n) ; 1\leq k\leq C(n) \}$.

\medskip

In a similar way we prove Property (2). 
From (KR5) it comes that all $h_{i}(n+1)$ are greater than 
$\sum_{j=1}^{C(n)} h_{j}(n)$. 
Consequently, from (LR) we obtain for all $1\leq k\leq C(n)$ and 
$1\leq k'\leq C(n)$
$$
h_{k}(n) 
\leq
\sum_{j=1}^{C(n)} h_{j}(n)
\leq
h_{i}(n+1) 
\leq  
L h_{k'}(n) .
$$ 
This ends the proof.
\end{proof}

From this lemma we deduce that the set $\{ M(n) ; n \geq 1 \}$ is 
finite. The following proposition, whose proof is left to the reader,
tells us this is in fact a necessary 
and sufficient condition to be linearly recurrent.

\begin{prop}
Let $(X,T)$ be a  minimal Cantor system.
The system $(X,T)$ is linearly recurrent if and only if there exist
a nested sequence of CKR 
partitions 
$(\P(n) ; n\in\NN )$, 
satisfying (KR1)-(KR6),
and a constant $K$ such that: for 
all $n\geq 1$ and all $(l,k)\in \{ 1, \dots , C(n) \} \times \{ 1 , 
\dots , C(n-1) \}$, 
$$
1\leq m_{l,k} (n) \leq K ,
$$
where $( M(n) = (m_{l,k}(n); 1\leq l \leq C(n) , 1\leq k \leq C(n-1)) ; 
n\geq 1 )$ be the associated sequence of matrices. 
\end{prop}

\subsection{Unique ergodicity and absence of strong mixing of linearly
  recurrent systems}

In this subsection $(X,T)$ is a linearly recurrent system with a 
nested sequence of CKR 
partitions 
$(\P(n)=\{T^{-j}B_{k}(n); 1\leq k\leq C(n), 0\le j<h_{k}(n) 
\};n\in\NN )$ 
satisfying (KR1)-(KR6) and (LR) with constant $L$. Let 
$( M(n) = (m_{l,k}(n); 1\leq l \leq C (n) , 1\leq k \leq C(n-1)) 
; n\geq 1 )$ be the associated sequence of matrices.

We notice that for each $T$--invariant probability measure 
$\mu$ and for 
every $n \ge 1$ 
and $1\leq k\leq C(n-1)$ we have
\begin{equation}
\label{B1}
\mu(B_k(n-1))=\sum_{l=1}^{C(n)}m_{l,k}(n)\,\mu(B_l(n)) 
\end{equation}
and
\begin{equation}
\label{B2}
\sum_{k=1}^{C(n)}h_k(n)\mu(B_k(n))=1\ .
\end{equation}

To prove that  linearly recurrent systems are uniquely ergodic we need the 
following lemma 
that is used through all this paper.

\begin{lemma}
\label{binf}
Let $\mu$ be an invariant measure of $(X,T)$. Then, for all $n\in 
\NN$ and $1\leq k\leq C(n)$ we have 
$$
h_{k}(n)\mu(B_{k}(n))\geq \frac{1}{L} .
$$
\end{lemma}
\begin{proof}
Fix $k$ with $1\leq k\leq C(n)$. By Equation~\eqref{B1}, since all the entries 
of 
$M(n+1)$ are positive, we get 
$$
\mu(B_k(n))\geq\sum_{\ell=1}^{C(n+1)}\mu(B_l(n+1))\ .
$$
By (LR), for every $l$ we have $h_k(n)\geq h_l(n+1)/L$ thus
$$
h_k(n)\mu (B_k(n))\geq \sum _{\ell=1}^{C(n+1)}
\frac{h_l(n+1)}{L}\mu(B_l(n+1))=\frac 1L\ .
$$
\end{proof}

\begin{prop}
Every linearly recurrent system is uniquely ergodic.
\end{prop}

\begin{proof}
Let  $(X,T)$ be a  linearly recurrent system.
Given a $T$--invariant probability measure $\mu$, we define the 
numbers
$$
\mu_{n,k}=\mu(B_{k}(n)),\ n\geq 0,\; 1\leq k\leq C(n)\ .
$$
These nonnegative numbers satisfy the relations
\begin{equation}
\label{eq:rel-mu}
\mu_{0,1}=1\text{ and }\mu_{n-1,k}=\sum_{l=1}^{C_n}\mu_{n,l}m_{l,k}(n)
\text{ for }
n\ge 1,\ 1\leq k\leq C (n-1)\ .
\end{equation}
In a matrix form: with $V(n)=(\mu_{n,1},\dots,\mu_{n,C(n)})$, we have
$V(n-1)=V(n)M(n)$.
Conversely, let the nonnegative numbers $(\mu_{n,k} ; n\geq 0,\ 
1\leq k\leq C(n)$) satisfy these conditions. As the partitions $\P 
(n)$ 
are clopen and span the topology of $X$, it is immediate to check 
that 
there exists a unique invariant probability measure $\mu$ on $X$ with
$\mu_{n,k}=\mu(B_{k}(n))$ for every $n\in \NN$ and $k\in \{ 1, \dots 
, C(n)\}$.

From Lemma~\ref{binf},
there exists a constant $\delta>0$ such that 
$$
\mu_{n,i}\geq 
\delta\mu_{n-1,k}
$$
for every $n\geq 1$ and $(i,k)\in \{ 1, \dots , C(n)\}\times \{ 1 , 
\dots , C(n-1) \}$,
and every invariant measure $\mu$. Without loss of generality we can 
assume $\delta<1/2$. 

Let $\mu,\mu'$ be two invariant measures, and $\mu_{n,k},\mu'_{n,k}$ 
be defined as above. We define
$$
 S_n=\max_k\frac{\mu'_{n,k}}{\mu_{n,k}} =\frac{\mu'_{n,i}}{\mu_{n,i}} 
\ ,\ 
 s_n=\min _k\frac{\mu'_{n,k}}{\mu_{n,k}}=\frac{\mu'_{n,j}}{\mu_{n,j}} 
\ , 
\hbox{ and } 
 r_n=\frac{S_n}{s_n}
$$
for some $i,j$.
For every $k\in \{ 1, \dots , C(n-1) \}$ we have:
\begin{align*}
\mu'_{n-1,k}
&=\sum_{l\neq j}\mu'_{n,l}m_{l,k}(n)+\mu'_{n,j}m_{j,k}(n)\\
&\leq S_n\sum_{l\neq j}\mu_{n,l}m_{l,k}(n)+s_n\mu_{n,j}m_{j,k}(n)\\
&=S_n\mu_{n-1,k}-(S_n-s_n)\mu_{n,j}m_{j,k}(n)
\leq \mu_{n-1,k}S_n-(S_n-s_n)\mu_{n,j}\\
&\leq \mu_{n-1,k}s_n\bigl(r_n(1-\delta)+\delta\bigr)\ .
\end{align*}

And in similar way, for every $k\in \{ 1, \dots , C(n-1) \}$ we have

$$
\mu'_{n-1,k} \geq \mu_{n-1,k}s_n\bigl(\delta r_n+(1-\delta)\bigr)\ .
$$
We deduce that
$$
 r_{n-1}\leq \phi(r_{n})\text{ where }
 \phi(x)=\frac {(1-\delta)x+\delta}{\delta x+(1-\delta)}\ .
$$

The function $\phi$ is increasing on $[0,+\infty)$ and tends to 
$(1-\delta)/\delta$ at $+\infty$. 
Writing $\phi^m=\phi\circ\dots\circ\phi$ ($m$ times), 
for every $n,m \in \NN$ we have
$1\leq r_n\leq\phi^m(r_{n+m})\leq \phi^{m-1}((1-\delta)/\delta)$. 
Taking the limit with $m\to+\infty$, we get $r_n=1$.
\end{proof}

From now on we call  $\mu$  the unique invariant measure on $(X,T)$.
Let $m\ge 1$ and $0\leq k\leq C(m)$. By unique ergodicity, 
$$
\frac 1N\#\{  0\leq j<N ; \ T^{-j}x\in B_k(m)\}\to \mu(B_k(m))
$$
uniformly as $N\to\infty$. But for $n>m$, $1\leq l\leq C(n)$, for 
every $x\in B_k(n)$ we have
$$
\#\{ 0\leq j<h_l(n) ; \ T^{-j}x\in B_k(m)\}= P_{l,k}(n,m)\ .
$$
We deduce:
\begin{equation}
\label{B3}
\max_{1\leq l\leq C(n)}\Bigl|\frac{P_{l,k}(n,m)}{h_l(n)}
-\mu\bigl(B_k(m)\bigr)\Bigr|\to 0\text{ as }n\to+\infty\ .
\end{equation}

\begin{prop}
Linearly recurrent systems are not strongly mixing. 
\end{prop}

\begin{proof}
Let $m$ be an integer such that $\mu (B_1(m)) < 1/L^2$ and  
for  $n >  m$ let $D (n) = B_1(m) \cap T^{h_{1}(n)} B_1(m)$. 
We  prove that 
$\varliminf_n \mu (D (n)) > \mu (B_1(m))^2 $ 
which will imply that $(X,T,\mu )$ is not strongly mixing.

For $n>m$ we write 
$$
E (n)=\{ 0\leq j < h_1(n-1) ; \ 
T^{-j} B_1(n-1) \subset B_1(m)\}\ .
$$
By hypothesis (KR6) we have  $B(n)\subset B_1(n-1)$ and
for $j\in E(n)$ we get 
$$
T^{{-j}-h_1(n)} B_1(n) \subset T^{-j} B(n) \subset T^{-j} B_1(n-1)
\subset B_1(m)
$$
and  $T^{-j} B_1(n)\subset D(n)$. It follows that
$\mu(D(n))\geq \#E(n)\cdot \mu(B_1(n))$.
But $\#E(n)=P_{1,1}(n-1,m)$ thus  $\#E(n)/ h_1(n-1)$ converges to
$ \mu(B_1(m))$ as $n\to+\infty$ by Equation~\eqref{B3}.
Therefore
\begin{align*}
\varliminf_n\mu(D (n))
&\geq\varliminf_n h_1(n-1)\mu(B_1(n))\mu(B_1(m))\\
&\geq\varliminf_n  \frac 1L  h_1(n)\mu(B_1(n))\mu(B_1(m))
& \text{ (by (LR))}\\
&\geq \frac 1{L^2}\mu(B_1(m))
& \text{ (by Lemma \ref{binf})}
\end{align*}
and the proof is complete.
\end{proof}

It is well-known that there exist substitution subshifts, and {\it a 
fortiori} linearly recurrent systems, which are weakly-mixing (see 
\cite{Qu}).

\section{Some conditions to be an eigenvalue}

In this section we suppose $(X,T,\mu )$ is linearly recurrent, 
that is to say $(\P (n) ; n \geq 0 )$ 
satisfies (KR1)-(KR6) and (LR) (with constant $L$). 
Let $(M(n) ; n\geq 1 )$ be its associated sequence of matrices.

We give a sufficient condition to be a continuous eigenvalue and a
necessary condition to be an eigenvalue.
We define for $n \ge 1$, $1 \le k \le C(n-1)$, $1 \le l \le C(n)$,
$$
J(n,k,l)
=
\bigl\{ 0\leq j<h_{l}(n) ; \ T^{-j}B_{l}(n)\subset B_{k}(n-1)
 \bigr\}
, \ 
J(n)=\bigcup_{\substack{1\leq k\leq C(n-1) \\ 1\leq l\leq 
C(n)}}J(n,k,l)\ .
$$
so that $\#J(n,k,l)= m_{l,k}(n)$.

\begin{prop}
\label{prop:condcont}
Let $\lambda\in\CC$ satisfy
$$
 \sum_{n=1}^\infty \max_{1\leq k\leq C(n)}
 \lvert\lambda^{h_{k}(n)}-1\rvert<\infty\ .
$$
Then $\lambda$ is a continuous eigenvalue of $(X,T,\mu)$.
\end{prop}

\begin{proof}
For every $n \in \NN$, let $f_n$ be the function on $X$ defined by
$$
f_n(x)=\lambda^{-j} \text{ for } x\in T^{-j}B_{k}(n),
$$
$1\leq k\leq C(n)$ and $0\leq j<h_{k}(n)$.

We compare $f_n$ and $f_{n-1}$. By construction, for every $x$,
$f_n(x)/f_{n-1}(x)$ belongs to the set 
$\{\lambda^{-j} ; j\in J(n)\}$.
But each integer in $J(n)$ is a sum of terms of the form 
$h_{k}(n-1)$, 
and this sum contains at most $L$ terms. We get
$$
||
f_n-f_{n-1}
||_\infty
\leq 
L\max_{1\leq k\leq C(n-1)}
 \lvert\lambda^{h_{k}(n-1)}-1\rvert\ .
$$
By hypothesis, the series $\sum_{n\geq 1} || f_n-f_{n-1}||_\infty$
 converges. Thus the sequence $(f_n ; n\in \NN)$ converges uniformly 
to a 
continuous function $f$, which is clearly an eigenfunction for 
$\lambda$.
\end{proof}

\begin{prop}
\label{prop:necessaire}
If $\lambda \in \CC$ is an eigenvalue of $(X,T,\mu)$ then
$$
\sum_{n=1}^\infty 
\max_{1\leq k\leq C(n)}
\bigl|\lambda^{h_{k}(n)}-1\bigr|^2<\infty\ .
$$
\end{prop}

\begin{proof}
We use the sets $J(n,k,l)$ defined above.

Assume that $\lambda=\exp(2i\pi \alpha)$, $\alpha\in \RR$, is an eigenvalue, 
and that $f$ is 
a corresponding eigenfunction of modulus $1$. For every $n \in \NN$, let 
$f_n$ be the conditional expectation of $f$ with respect to the 
$\sigma$-algebra spanned by $\P_n$. For $1\leq k\leq C(n)$, $f_n$ is 
constant on $B_{k}(n)$, and we write $c(n,k)$ this constant.

The sequence $(f_n ; n\in \NN)$ is a martingale (\cite{Do}), and 
converges to $f$ in $L^2(\mu)$. Moreover the functions $f_n-f_{n-1}$, 
$n\geq 1$, are mutually orthogonal in  $L^2(\mu)$, hence we have
\begin{equation}
\label{eq:mart}
\sum_{n=1}^\infty ||{f_n-f_{n-1}}||^2_2<\infty\ 
\end{equation}
(see \cite{Do} for the details).

We fix $n\ge 1$, $1\leq l\leq C(n)$ and $1\leq k\leq C (n-1)$, and we 
choose some  $j\in J(n,k,l)$. 
Looking at the structure of the towers, we see 
that $j+h_{k}(n-1)\leq h_{l}(n)$.
For $0\leq p<h_{k}(n-1)$ we have $j+p<h_{l}(n)$, 
and $T^{-(j+p)}B_{l}(n)\subset T^{-p}B_{k}(n-1)$. 
For $x\in T^{-(j+p)}B_{l}(n)$, we have 
$f_n(x)=\exp(-2i\pi (j+p)\alpha)c(n,l)$ and 
$f_{n-1}(x)=\exp(-2i\pi  p \alpha)c(n-1,k)$.
We get
$$
||f_n-f_{n-1}||^2_2\geq h_{k}(n-1)\mu (B_{l}(n))
\lvert \exp(-2i\pi j\alpha)c(n,l)-c(n-1,k)\rvert^2\ .
$$
By Lemma \ref{binf} and (LR), $h_{k}(n-1)\mu (B_{l}(n)) \geq L^{-2}$ 
, and from Equation~\eqref{eq:mart} we get
\begin{equation}
\label{eq:mart2}
\sum_{n=1}^\infty \max_{1\leq l\leq C(n)}  \max_{1\leq k\leq C(n-1)}
\max_{j\in J(n,k,l)}
\lvert \exp(-2i\pi j \alpha )c(n,l)-c(n-1,k)\rvert^2<\infty\ .
\end{equation}

We use this bound first with $k=1$ and an arbitrary $l$. By (KR6),
$0\in J(n,1,l)$ and from Equation~\eqref{eq:mart2} we get 
$$
\sum_{n=1}^\infty \max_{1\leq l\leq C(n)} \lvert 
c(n,l)-c(n-1,1)\rvert^2<\infty\ .
$$

Using this three times, we get
\begin{equation}
\label{mart3}
\sum_{n=1}^\infty  \max_{1\leq l\leq C(n)}\max_{1\leq k\leq C(n-1)}
|c(n,l)-c(n-1,k)|^2<\infty\ .
\end{equation}

For each $n\in \NN$ and $1\leq k\leq C(n)$, the function $|f_n|$ is constant 
and equal to $|c(n,k)|$ on the $k$-th tower of $\P (n)$. By Lemma~\ref{binf},
the measure of this tower is not less than $1/L$. 
Since $\lVert f_n\rVert_2\to \lVert f\rVert_2=1$, we get that 
$\inf_k|c(n,k)|$ converges to $1$ when $n\to+\infty$.
Hence from Equations~\eqref{eq:mart2} and~\eqref{mart3}, we get
$$
\sum_{n=1}^\infty \max_{1\leq l\leq C(n)}  \max_{1\leq k\leq C(n-1)}
\max_{j\in J(n,l,k)}
\lvert \exp(-2i\pi j\alpha)-1\rvert^2<\infty\ .
$$
We use this bound with two consecutive elements of the same set 
$J(n,l,k)$ and get the announced result.
\end{proof}

The following sufficient condition for weak  mixing follows from 
Proposition~\ref{prop:necessaire}.
A similar condition appears in \cite{FHZ}.  

\begin{coro}
\label{FHZ}
For every $n\in \NN$, let 
$$
K_n=\inf\bigl\{  |h_i(n) - h_j(n)|: 1\leq i,j\leq C(n),\
h_i(n) \neq h_j(n)\bigr\}
$$
and let $K=\varliminf_n K_n$.
If $K$ is finite, then $(X,T,\mu)$ has at most $K$ eigenvalues.
In particular, if $K=1$ then  this system is weakly mixing.
\end{coro}

Now we restate Proposition~\ref{prop:condcont} and 
Proposition~\ref{prop:necessaire}
in terms of matrices. 

\medskip

{\bf Notation.}
For every real number $x$ we write $\vvert x\vvert$ for the distance of $x$ 
to the nearest integer.
For a vector $V=(v_1,\dots,v_m)\in \RR^m$, we write 
$$
\Vert V\Vert=\max_{1\leq j\leq m}|v_j|\text{ and }
\vvert V\vvert=\max_{1\leq j\leq m}\vvert v_j\vvert\ .
$$
We use similar notations for real matrices.
With these notations, the two preceding Theorems can be written as 
follows.

\begin{theo}
\label{th:condsuffmat}
Let $\alpha \in \RR$ and $\lambda=\exp(2i\pi \alpha)$. 
\begin{enumerate}
\item 
If $\lambda$ is an eigenvalue of $(X,T,\mu)$ then 
$\displaystyle \sum_{n\ge 1} \vvert \alpha P(n)  H(1) \vvert^2 
<\infty$. 
\item
If 
$\displaystyle \sum_{n\ge 1}\vvert  \alpha P(n)  H(1) \vvert <\infty$ 
then $\lambda$ is a continuous eigenvalue of 
$(X,T,\mu)$.
\end{enumerate}
\end{theo}

\begin{prop}
\label{description}
Let $\alpha\in\RR$ and $\lambda = \exp (2i\pi \alpha )$.  

If $\lambda$ is an eigenvalue of $(X,T,\mu)$ then it satisfies at 
least one of the two following properties:
\begin{enumerate}
\item
$\alpha$ is rational, with a denominator dividing
$\gcd(h_i(m): 1\leq i\leq C(m))$ for some $m \in \NN$.
In this case $\lambda$ is a continuous eigenvalue.
\item
There exist $m \in \NN$ and integers $w_j$, $1\leq j\leq C(m)$, such that
$
\displaystyle\alpha = \sum_{j=1}^{ C(m)} w_j \mu (B_j(m)) .
$
\end{enumerate}
Moreover, if $\alpha $ is rational, with a denominator dividing
$\gcd(h_i(m): 1\leq i\leq C(m))$ for some $m \in \NN$, then $\lambda $ is
 an eigenvalue of $(X,T,\mu )$. 
\end{prop}

The proof of Proposition~\ref{description} needs the following lemma.

\begin{lemma}
\label{tendto0}
Let $u$ be a real vector such that $\vvert P(n)u\vvert\to 0$ as 
$n\to+\infty$.
Then there exist $m \in \NN$, an integer vector $w$ and a real vector $v$ 
with
$$
P(m)u=w+v\text{ and } \Vert P(n,m)v\Vert\to 0\text{ as 
}n\to+\infty\ .
$$
\end{lemma}

\begin{proof}
By hypothesis, for every $n \in \NN$ we can write $P(n)u=v_n+w_n$, where 
$w_n$ is an integer vector and $v_n$ a real vector with 
$\Vert v_n\Vert\to 0$ as $n\to+\infty$.
Since all the matrices $M(m)$ belong to a finite family, 
$\Vert M(m)v_m-v_{m+1}\Vert$ converges to $0$ as $m$ goes to 
infinity. 
But for every $m \in \NN$ we have $P(m+1)u=M(m+1)P(m)u$, thus 
$$
M(m)v_m-v_{m+1}=-M(m)w_m+w_{m+1}
$$
and $M(m)v_m-v_{m+1}$ is an integer vector.
Therefore the sequence $(M(m)v_m-v_{m+1})$ is eventually zero.
There exists $m \in \NN$ such that $v_n=P(n,m)v_m$ for every $n>m$.
The vectors $v=v_m$ and $w=w_m$ satisfy the announced properties.
\end{proof}

\begin{proof}[Proof of Proposition~\ref{description}]
Let $u=\alpha H(1)$. 
Since $\lambda$ is an eigenvalue, 
$\vvert P(n) u\vvert\to 0$ as $n  \to \infty$ by Theorem~\ref{th:condsuffmat}. 
Let $m,v$ and $w$ be as in Lemma~\ref{tendto0}.
We recall that $P(m)u=\alpha H(m)$. 
We distinguish two cases. 

First we assume that $v=0$. Then $\alpha H(m)$ is equal to the 
integer 
vector $w$, and $\alpha$ is rational with a denominator dividing
$\gcd(h_i(m): 1\leq i\leq C(m))$. 
For $n\geq m$ the vector $\alpha H(m)$ has integer entries, thus
$\vvert\alpha H(m)\vvert=0$ and $\lambda$ is a continuous eigenvalue 
by Theorem~\ref{th:condsuffmat}.

Now suppose $v\not = 0$. For $n>m$ we have
\begin{align*}
\sum_{k=1}^{C(m)}\mu(B_k(m))v_k
&=\sum_{k=1}^{C(m)}\sum_{l=1}^{C(n)}P_{l,k}(n,m)
\,\mu(B_l(n))v_k
& \text{ by (2.3)}\\
&=\sum_{l=1}^{C(n)}(P(n,m)v)_l\, \mu(B_l(n))
\leq \Vert P(n,m)v\Vert
\end{align*}

and the last term converges to $0$ as $n\to+\infty$, thus 
$$
\sum_{k=1}^{C(m)}\mu(B_k(m))v_k=0\ .
$$
As $w=\alpha H(m)-v$, that is, $w_j=\alpha h_j(m)-v_j$ for $1\leq 
j\leq C(m)$, we get
$$
\sum_{j=1}^{C(m)}w_j\mu(B_j(m))
=\alpha\sum_{j=1}^{C(m)}h_j(m)\mu(B_j(m))=\alpha\ .
$$
The rest of the proof is left to the reader.
\end{proof}

\section{Examples}

\label{sectionexample}

We study some examples where we can explicitly say 
that the eigenfunctions are continuous or there do not exist 
non trivial eigenvalues. 
We keep the notations and 
hypotheses of the preceding section.

\subsection{Example 1: The sequence $(M(n) ; n\geq 2)$ is ultimately constant}

Let $(M(n) ; n\geq 1)$ be the sequence of matrices associated to 
the linearly recurrent system $(X,T,\mu)$. 
We say that $(X,T,\mu)$ has  \emph{a stationary sequence of matrices} 
if there exist a square matrix $M$ and an integer $n_0\in \NN$ 
such that $M(n) = M$ for all $n\geq n_0$. 
Without loss we can assume that $n_0 = 2$. 
We have $P_n = M^{n-1}$ for $n\geq 2$.

Substitution subshifts and odometers with constant base belong to 
 the family of linearly recurrent systems with a stationary 
sequence of matrices (see \cite{DHS}).
The following lemma was used in \cite{Ho} to prove that eigenvalues 
of substitution subshifts are continuous.

\begin{lemma}
\label{statexp}
Let $M$ be a matrix with integer entries.
If $u$ is  a real vector such that 
$\lVert M^n u\rVert\to 0$ when $n\to\infty$, then
the convergence is exponential, i.e., there exist $ 0\leq r <1  $ and 
a constant $ K $  such that
$\lVert M^n u\rVert \leq K r^{n} $ for all $n\in \NN $.
\end{lemma}

From this Lemma and Theorem~\ref{th:condsuffmat} we get:

\begin{prop}
Let $(X,T,\mu)$ be a linearly recurrent Cantor system with a 
stationary sequence of matrices.
Then every eigenfunction of this system is continuous.

 Moreover all the linearly recurrent Cantor systems  with the same 
stationary sequence of matrices have the same 
eigenvalues.
\end{prop}

\subsection{Example 2: A family of weakly mixing systems} 

We build a family of linearly recurrent systems which are the 
Cantor 
analogues of interval exchange transformations considered in 
\cite{FHZ} (Theorem 2.2). 
Let $N$ be a positive integer. Let  $(M(n) ; n \geq 2)$ be a 
sequence of matrices in the  family
$$\left\{
\left[ 
\begin{array}{ccc}
l   & k-1 & 1\\
l-1 & k   & 1\\
l-1 & k-1 & 1
\end{array}
\right] 
,
\
\left[ 
\begin{array}{ccc}
l-1 & k   & 1\\
l   & k-1 & 1\\
l   & k   & 1
\end{array}
\right] 
\ :\ 
 1\leq l,k \leq N\right\}
$$
and let $(X,T)$ be a linearly recurrent system with this sequence of 
matrices.
For any  $n$ and any $v = (v_1 , v_2 , v_3 )^t$, the vector $u=M(n)v$ 
satisfies $|u_1 - u_2| = |v_1 - v_2|$.

Consequently if we suppose $|h_{1}(1) - h_{2}(1)| = 1$ 
then it follows by induction that for all $n\geq 1$ we have $|h_{1}(n) 
- h_{2}(n)| = 1$. 
By Corollary~\ref{FHZ} the system $(X,T,\mu)$ does not have non trivial
eigenvalues, i.e. it is weakly mixing.

\subsection{Example 3: The sequence $(M(n) ; n\geq 1 )$ has infinitely 
many rank 1 matrices}

\begin{prop}
Let $(X,T,\mu)$ be a linearly recurrent Cantor system and let the associated sequence of matrices be $(M(n);n \ge 1)$.
Suppose that $M(n)$ has rank one for infinitely many values of $n$.
Then $\lambda = \exp (2i\pi \alpha )$ is an eigenvalue of $(X,T,\mu)$ if and only 
if $\alpha$ is rational with a denominator equal to  
$\gcd( h_i(m):1\leq i\leq C(m))$ for some $m \in \NN$. 
Moreover every eigenfunction is continuous.
\end{prop}

\begin{proof}
Let $\lambda = \exp (2i\pi \alpha )$ be an eigenvalue of 
$(X,T,\mu)$. As in the proof of Proposition~\ref{description}, 
we write $u=\alpha H(1)$, and  take  $m,v$ and $w$ 
as given by  Lemma~12.
We choose $n>m$ such that $M(n)$ is of rank $1$.
We have
\begin{multline*}
\ker(P(n,m))
\subset\bigl\{x\in\RR^{C(m)}: 
\Vert P(l,m)x\Vert\to_{l\to \infty} 0\bigr\}\\
\subset \bigl\{x\in\RR^{C(m)}:
\sum_{j=1}^{C(m)}x_j\mu(B_j(m))=0\bigr\}
\end{multline*}
where the last inclusion follows from the proof of 
Proposition~\ref{description}. 
The third of these three spaces is of codimension $1$ in $\RR^{C(m)}$.
The matrix $P(n,m)$ is not zero  and has a rank $\leq 1$, 
thus the first of these three linear spaces is of codimension $1$ 
also, and these spaces are actually equal. Since $v$ belongs to the 
second space, it belongs to the first one, and $P(n,m)v=0$. 
Thus $\alpha H(n)=P(n,m)w$ and it has integer entries.
We conclude as in the proof of Proposition~\ref{description}.
\end{proof}

\subsection{Example 4: $2\times 2$ matrices with determinant equal 
to $\pm 1$}

Here we assume that  the matrices $( M(n) $; $n\geq 2)$ of the 
linearly recurrent system $(X,T,\mu)$ are $2\times 
2$ matrices 
with entries in $\{1, \dots ,L\}$ and determinant $\pm 1$. 
We assume also that $h_1(1)=h_2(1)=1$.
We set 

$$
P_{n}=\left[ 
\begin{array}{cc}
x_{n} & y_{n}\\
z_{n} & w_{n}
\end{array}\right]
\text{ and }\Delta_n=\det(P(n))=\pm 1 \ .
$$
We notice that
$$
x_n+y_n=h_1(n)\text{ and }z_n+w_n=h_2(n)\ .
$$
Since for every $n\ge 2$ all the entries of $M(n)$ are positive, we get 
easily 
that $h_1(n)\geq 2^{n-1}$ and $h_2(n)\geq 2^{n-1}$.

By Equation~\eqref{B3}, as $n\to\infty$, we have
$$
\frac{x_n}{h_1(n)}\to \mu(B_1(1))\ ,\ 
\frac{y_n}{h_1(n)}\to \mu(B_2(1))\ ,\ 
\frac{z_n}{h_2(n)}\to \mu(B_1(1))\ ,\ 
\frac{w_n}{h_2(n)}\to\mu(B_2(1))\ .
$$

\begin{lemma}
\label{lemme18}
For a vector  $ v\in \RR ^2 $ we have 
$\lim _{n\rightarrow \infty }\Vert P(n)v\Vert =0 $ if and only if 
$v$ is collinear with $\bigl(\mu(B_2(1)), -\mu(B_1(1))\bigr)$. 
In this case  the convergence is exponential.
\end{lemma}

\begin{proof}
 We write $\mu_1=\mu(B_1(1))$ and $\mu_2=\mu(B_2(1))$; clearly 
$\mu_1+\mu_2=1$.
 
Let $v\in\RR^2$.
As in the proof of Proposition~\ref{description}, if $\Vert 
P(n)v\Vert\to 0$ then 
$\mu_1v_1+\mu_2v_2=0$, and $v$ is collinear with  $(\mu_2,-\mu_1)$.
It remains to show that $\Vert P(n)v\Vert\to 0$ exponentially when 
$v$ is collinear with  $(\mu_2,-\mu_1)$. Obviously we can restrict 
ourselves to the case where  these vectors are equal.

We check that $x_{n+1}h_1(n)-x_nh_1(n+1)=-m_{1,2}(n+1)\Delta_n$, and 
deduce that
$$
\Bigl| \frac{x_{n+1}}{h_1(n+1)}-\frac{x_n}{h_1(n)}\Bigr|
=\Bigl|\frac{m_{1,2}(n+1)\Delta_n}{h_1(n+1)h_1(n)}\Bigr|
\leq \frac L{h_1(n+1)h_1(n)}\ .
$$

Since $x_n/h_1(n)\to \mu_1$ as $n \to \infty$, we get
$$
\Bigl|\frac {x_n}{h_1(n)}-\mu_1\Bigr|\leq 
\sum_{i=n}^\infty \frac L{h_1(i+1)h_1(i)}
\leq \frac C{2^nh_1(n)}
$$
for some constant $C$. From $x_n+y_n=h_1(n)$ and $\mu_1+\mu_2=1$, 
we obtain
$$
|x_n\mu_2-y_n\mu_1|=|x_n-h_1(n)\mu_1|\leq C2^{-n}\ .
$$
In a similar way, we have $|z_n\mu_2-w_n\mu_1|\leq C2^{-n}$. We conclude
$\Vert P(n)v\Vert\leq C2^{-n}$.
\end{proof}

\begin{prop}
Let $(X,T,\mu)$ be a linearly recurrent Cantor system and let the
 associated sequence of matrices $(M(n);n \ge 1)$ be $2\times 
2$ matrices 
with entries in $\{1, \dots ,L\}$ and determinant $\pm 1$. 
Suppose also that $h_1(1)=h_2(1)=1$.
Let $\alpha\in\RR$ and $\lambda =\exp(2i\pi\alpha)$.
Then $\lambda$ is an  eigenvalue 
of $(X,T,\mu)$ if and only if $\alpha$ belongs to  the set
$\{p_1\mu(B_1(1))+p_2\mu(B_2(1)): p_1,p_2 \in \ZZ \}$. 
Moreover every eigenfunction of $(X,T,\mu)$ is continuous. 
\end{prop}

\begin{proof}
Let $\lambda = \exp (2i \pi \alpha )$ be an eigenvalue of 
$(X,T,\mu)$. 
By Theorem~\ref{th:condsuffmat} we have $\vvert \alpha 
P(n)H(1)\vvert\to 0$ as $n \to \infty$.
We take $u=\alpha H(1)$, and let $v,w$  and $m$  as given by 
Lemma~\ref{tendto0}. We write $v'=P(m)^{-1}v$ and  $w'=P(m)^{-1}w$.
We have $\alpha H(1)=v'+w'$, $\Vert P(n)v'\Vert\to 0$ as $n\to \infty$, and
the vector $w'$ is an integer vector because the matrix $P(m)$ has 
integer entries and $|det ( P(m))|=1$. By Lemma~\ref{lemme18}, there exists  $k\in\RR$ with 
$v'_1=k\mu(B_2(1))$ and $v'_2=-k\mu(B_1(1))$.
Since $u_1=u_2=\alpha$ and $\mu(B_1(1))+\mu(B_2(1))=1$, we have 
$k=v'_1-v'_2=-w'_1+w'_2\in\ZZ$, and $\alpha$ has the announced form.

Conversely, if $\alpha=p_1\mu(B_1(1))+p_2\mu(B_2(1))$ for some 
integers $p_1$ and $p_2$, the vector $\alpha H(1)$ can be written as 
the sum of an integer vector $w$ and a vector $v$ collinear with 
$\bigl(\mu(B_2(1)), -\mu(B_1(1))\bigr)$. By Lemma~\ref{lemme18},
$\Vert P(n) v \Vert\to 0$ exponentially as $n\to \infty$, thus $\vvert \alpha 
P(n) H(1) \vvert\to 0$ exponentially as $n\to \infty$, and $\alpha$ is a continuous 
eigenvalue by Theorem~\ref{th:condsuffmat}.
\end{proof}

\subsection{Example 5: Two commuting matrices}

Let $(X,T,\mu)$ be a linearly recurrent system with  $H (1) = (1,1)^t$.
We assume that each matrix 
$M(n)$, $n \ge 2$, is one of the two following ones:
\[
A=\left[ \begin{array}{cc}
5 & 2\\
2 & 3
\end{array}\right]
\hbox{ and } 
B=\left[ 
\begin{array}{cc}
2 & 1\\
1 & 1
\end{array}
\right] .
\]
We notice that the matrices $A$, $B$ have the same eigenvectors and 
commute.
We write $\alpha_1,\alpha_2$ for the eigenvalues of $A$, with 
$\alpha_1>\alpha_2>1$, and $\beta_1,\beta_2$ for the eigenvalues of 
$B$, with $\beta_1>1>\beta_2>0$.
We set $\delta=-\log(\beta_2)/ \log(\alpha_2/\beta_2)$. 
For every $n$ we write $a_n$ (respectively $b_n$) for  the number of 
occurrences of $A$ (respectively $B$) in the  sequence $M(2),\dots,M(n)$.
For every $n$ the eigenvalues of $P(n)$ are 
$\alpha_1^{a_n}\beta_1^{b_n}$ and $\alpha_2^{a_n}\beta_2^{b_n}$, and 
we have
$$
H(n)=\alpha_1^{a_n}\beta_1^{b_n}u_1+ \alpha_2^{a_n}\beta_2^{b_n}u_2
$$
where $u_1$, $u_2$ are two non-zero vectors.

It is not difficult to show by induction that $\gcd(h_m(1),h_m(2))=1$ 
for every $m$. 
Assume first that $\limsup  a_n/n >\delta$. 
There does not exist any $m\geq 2$ and  $0\neq v\in\RR^2$ with 
$\lVert P(n,m)v\rVert\to 0$. By Lemma~\ref{tendto0}, there does not exist  
any $u\in\RR^2\setminus\ZZ^2$ such that $\vvert P(n) u\vvert\to 0$. 
By Theorem~\ref{th:condsuffmat} the system is weakly mixing.

Let us assume now that  $\limsup a_n/n <\delta$. 
If $m$ is an integer and $v$ a vector such that $\lVert P(n,m) 
v\rVert\to 0$ as $n\to+\infty$, then $v$ is collinear with $u_2$ and 
the convergence is exponential. It then follows from 
Lemma~\ref{tendto0} and Theorem~\ref{th:condsuffmat} that every 
eigenfunction is continuous. Moreover, there exist  real numbers 
$\alpha,s$ with  $\alpha\notin\ZZ$  such that 
the vector $(\alpha,\alpha)-sv_2$ belongs to $\ZZ^2$. 
Then $\exp(2\pi i\alpha)$ is an eigenvalue, and the system is not 
weakly mixing.
We summarize:

\begin{itemize}
\item
If $\limsup\, a_n/n >\delta$ then the system is weakly mixing.
\item
If $\limsup\, a_n/n <\delta$ then the system is not weakly mixing, and 
all of its eigenfunctions are continuous.
\end{itemize}

\section{A random linearly recurrent system}
\label{sec:ingeneral}
The preceding examples lead to the somewhat vague intuition that 
for ``most'' of the linearly recurrent systems the eigenfunctions are 
continuous. We test here this guess by building random linearly 
recurrent systems in a natural and relatively general way.

Let $\M$ be a finite set of matrices, not assumed to be of the same 
size, and let $\M^{\NN}$ be endowed with 
the shift $S$ and with the product topology. 
We write $\Omega$ for the 
subset of $\M^{\NN}$ consisting in sequences $(M_n)$ such that the 
product $M(n+1)M(n)$ is defined for every $n$ and we assume that 
$\Omega$ is not empty. $\Omega$ is a closed $S$-invariant subset of 
$\M^{\NN}$, and $(\Omega,S)$ is a subshift of finite type. 
We write $M=(M(n))$ for an element of $\Omega$.

For every sequence $M$ in $\Omega$ let $(X_M,T)$ be a 
linearly recurrent system associated to this sequence.
By choosing a  probability measure $\nu$ on $\Omega$ we get a random 
linearly recurrent system. We henceforth assume that $\nu$ is 
invariant and ergodic under $S$.
We show now that under some natural hypothesis the eigenfunctions of 
$X_M$ are continuous for $\nu$-almost every $M$.

Let $k$ be  the maximum size of the elements 
of $\M$. By completing each  of these elements by zero entries we can 
consider them as $k\times k$ matrices. We choose a norm on $\RR^k$ and 
a norm on $\L(\RR^k)$. Let  $A$ be the map from 
$\Omega$ to $\L(\R^k)$ which maps each sequence $M$ to $M(2)$.
The function $M\mapsto \log^+(\lVert A(M)\rVert)$ is bounded and thus 
belongs to $L^1(\nu)$.

By Oseledets Theorem (see~\cite{Wa}) there exists a measurable 
subset $B$ of $\Omega$, invariant under $S$ and of full measure, such 
that for every $M\in B$ the limit  
\begin{equation}
\label{eq:oseledets}
\lim_{n\to +\infty}\frac 1n\log 
\lVert A(S^{n-1}M)\circ \cdots \circ A(SM)\circ A
(M)(v)\|
\end{equation}
exists in $\RR\cup\{\pm\infty\}$.

We say that  $(\Omega,S,\nu)$ is \emph{hyperbolic} if the set $B$ can 
be chosen so that 
for every $M\in B$ and every $v\in\RR^k$ this limit is non-zero.
 Henceforth we  assume that this condition holds and show that  
for any $M\in B$  the eigenfunctions 
of $X_M$ are continuous.

Let $m$ be an integer and $v\in\RR^{C(m)}$ a vector such that
$\lVert P(n,m)v\rVert\to 0$ as $n\to +\infty$.
As $B$ is invariant under $S$, the sequence $(M(n):n>m)$ belongs to 
$B$ and the limit~\eqref{eq:oseledets} exists. By hypothesis this 
limit can not be positive, and it is non zero by hyperbolicity, then 
it is negative. It follows that  $\lVert P(n,m)v\rVert\to 0$ 
exponentially. 

By Theorem~\ref{th:condsuffmat} and Lemma~\ref{tendto0}, every eigenfunction of $X_M$ is 
continuous for $M\in B$ and thus $\nu$-almost everywhere.

\section{Questions}

Is it true that all eigenvalues of linearly recurrent systems are 
continuous ? 
If it is not true, is the result of Section~\ref{sec:ingeneral} true 
without 
the assumption of ``hyperbolicity'' ? 
If the answer is again negative could we find some necessary and 
sufficient condition to have only 
continuous eigenvalues ?
In fact, it seems that the existence of non continuous eigenfunctions is
not only a property of the sequence of matrices, but it depends on other
elements of the dynamics.

\medskip

{\bf Acknowledgments.} The authors acknowledge financial support
from FONDAP program in Mathematical Modeling, 
FONDECYT 1010447, ECOS-Coni\-cyt C99E10 and Programa Iniciativa Cientifica Milenio P01-005. This paper was mainly 
written when the second author was ``d\'etach\'e au CNRS" at the 
Centro de Modelamiento Matem\'atico UMR 2071.

\end{document}